\documentclass{amsart}
\usepackage{amsthm}
\usepackage{tikz}
\usepackage{mathrsfs}
\usetikzlibrary{calc,decorations.markings}

\newtheorem{thm}{Theorem}[section]
\newtheorem{lem}{Lemma}[section]
\newtheorem{prop}{Proposition}[section]
\newtheorem{coro}{Corollary}[section]
\newtheorem{defi}{Definition}[section]

\newtheorem*{notation}{Notation}

\def\A{\mathcal{A}}
\def\half{\frac{1}{2}}
\def\Re{\mathrm{Re}}
\def\Im{\mathrm{Im}}

\def\d{\mathrm{d}}

\def\A{\mathcal{A}}
\def\M{\mathcal{M}}

\begin{document}
\author{Kamalakshya Mahatab}
% \address{Institute of Mathematical Sciences, HBNI, 
% CIT Campus, Taramani, Chennai 600113, India}
\address{NTNU, Trondheim, Norway}
\email[Kamalakshya Mahatab]{accessing.infinity@gmail.com}
\author{Anirban Mukhopadhyay}
\address{Institute of Mathematical Sciences, HBNI, 
CIT Campus, Taramani, Chennai 600113, India}
\email[Anirban Mukhopadhyay]{anirban@imsc.res.in}

\begin{abstract}
For a fixed $\theta\neq 0$,  we define the twisted divisor function
\[
 \tau(n, \theta):=\sum_{d\mid n}d^{i\theta}\ .
\]
In this article we consider the error term $\Delta(x)$ in the following asymptotic
formula
\[ \sum_{n\leq x}^*|\tau(n, \theta)|^2=\omega_1(\theta)x\log x + \omega_2(\theta)x\cos(\theta\log x)
+\omega_3(\theta)x + \Delta(x),\]
where $\omega_i(\theta)$ for $i=1, 2, 3$ are constants depending only on $\theta$.
We obtain
\[\Delta(T)=\Omega\left(T^{\alpha(T)}\right) \text{ where } 
\alpha(T) =\frac{3}{8}-\frac{c}{(\log T)^{1/8}} \text{ and } c>0,\]
along with an $\Omega$-bound for the Lebesgue measure of the set of points where the above
estimate holds.
\end{abstract}


\title{Omega theorems for the twisted divisor function}
\maketitle

\vskip 4mm

\section{Introduction}

For an arithmetical function $f(n)$ we write
\[ \sum_{n \le x}f(n)=M(x)+\Delta(x),\]
where $M(x)$ is the main term and $\Delta(x)$ is the error satisfying $\Delta(x)=o(M(x))$.
An $\Omega$-estimate for $\Delta(x)$ helps us understand the magnitude of fluctuation of error and thereby 
measures the sharpness of an upper bound for error.
 
In \cite{BaluRamachandra1} and \cite{BaluRamachandra2}, Balasubramanian and Ramachandra introduced
a method to obtain a lower bound for 
\[ \int_T^{T^{\mathfrak b}} \frac{ |\Delta(x)|^2}{x^{2\alpha+1}} \d x\] 
in terms of the second moment of the corresponding Dirichlet series $D(s)$,
for some  $\mathfrak b>0$ and $\alpha>0$. A nondecreasing lower bound gives 
\[ \Delta(x)=\Omega(x^{\alpha-\epsilon}) \quad \text{for any } \epsilon>0 .\]
In these papers, they considered the error terms in asymptotic formulas for partial sums of certain arithmetic functions such as sum of square-free divisors and counting function for non-isomorphic abelian groups. 
This method requires the Riemann Hypothesis to be assumed in certain cases. Balasubramanian, Ramachandra and Subbarao \cite{BaluRamachandraSubbarao} modified this technique to apply on error term in the asymptotic 
formula for the counting function of $k$-full numbers without assuming Riemann Hypothesis. This method has been used by several authors including
\cite{Nowak} and \cite{srini}.

For a fixed $\theta\neq 0$,  we consider 
\begin{equation}\label{eq:tau-n-theta_def}
 \tau(n, \theta)=\sum_{d\mid n}d^{i\theta}\ .
\end{equation}
%%Addition 1
Note that %$\tau(n, \theta)$ is a multiplicative function, and 
\[\sum_{\substack{d|n\\a\leq\log d\leq b}}^*1=\frac{1}{2\pi}\int_{-\infty}^{\infty}\tau(n, \theta)\frac{e^{-ib\theta}-e^{-ia\theta}}{-i\theta}\d \theta,\]
where $*$ denotes that if $e^a|n$ or $e^b|n$ then their contribution to the sum is $\half$. So in principle we can restate questions on distribution of divisors of $n$ in terms of $\tau(n, \theta)$ and 
can take advantage of the multiplicative structure of $\tau(n, \theta)$.
This function is used in \cite{DivisorsHallTenen} to measure the clustering of divisors. In this paper we will study 
the Dirichlet series of $|\tau(n, \theta)|^2$, which can be expressed in terms of the Riemann zeta function as 
\begin{equation}\label{eq:dirichlet_series_tauntheta}
 D(s)=\sum_{n=1}^{\infty}\frac{|\tau(n, \theta)|^2}{n^s}=\frac{\zeta^2(s)\zeta(s+i\theta)\zeta(s-i\theta)}{\zeta(2s)}
 \quad\quad  \text{for}\quad \Re(s)>1.
\end{equation}
In \cite[Theorem 33]{DivisorsHallTenen}, Hall and Tenenbaum proved that
\begin{equation}\label{eq:formmula_tau_ntheta}
 \sum_{n\leq x}|\tau(n, \theta)|^2=\omega_1(\theta)x\log x + \omega_2(\theta)x\cos(\theta\log x)
+\omega_3(\theta)x + \Delta(x),
\end{equation}
where $\omega_i(\theta)$s are explicit constants depending only on $\theta$ and
\begin{equation}\label{eq:upper_bound_delta}
 \Delta(x)=O_\theta(x^{1/2}\log^6x).
\end{equation}
Here the main term comes from the residues of  $D(s)$ at $s=1, 1\pm i\theta $.
All other poles of $D(s)$ come from the zeros of $\zeta(2s)$. Using a pole on the line $\Re(s)=1/4$, 
Landau's method gives
\[\Delta(x)=\Omega_{\pm}(x^{1/4}).\]
In \cite{measure_asp}, we show that
\begin{align*}
 \mu \left( \A_j\cap [T, 2T]\right)=\Omega\left(T^{1/2}(\log T)^{-12}\right) \quad \text{ for } j=1, 2,
\end{align*}
where
\begin{align*}
 &&\A_1&=\left\{x: \Delta(x)>(\lambda(\theta)-\epsilon)x^{1/4}\right\}&\\
 &\text{and}&\A_2&=\left\{ x : \Delta(x)<(-\lambda(\theta)+\epsilon)x^{1/4}\right\},&\\
\end{align*}
for any $\epsilon>0$ and $\lambda(\theta)>0$. 
Moreover, under Riemann Hypothesis, we obtained
\[\mu\left(\A_j\cap [T, 2T]\right) =\Omega\left(T^{3/4-\epsilon}\right),\quad \text{ for } j=1, 2 \]
and for any $\epsilon>0$.

Adopting the method of Balasubramanian, Ramachandra and Subbarao in case of this twisted divisor function, we derive the following theorem.
\begin{thm}\label{omega_integral}
For any $c>0$, there exist constants $K(c)>0$ and $T(c)>0$ such that for all $T \ge T(c)$, we get 
\begin{equation}\label{lb-increasing}
\int_T^{\infty} \frac{|\Delta(x)|^2}{x^{2\alpha+1}}e^{-2x/y} \d x
\geq K(c)\exp\left( c(\log T)^{7/8} \right),
\end{equation}
where
\[ \alpha=\alpha(T)=\frac{3}{8} -\frac{c}{(\log T)^{1/8} } \ \text{ and } \ y=T^{\mathfrak b}
\  \ \text{ for }\mathfrak b\geq 80. \]
\end{thm}
\noindent 
In particular, this implies
\[ \Delta(x)=\Omega \left(x^{3/8}\exp\left(-c(\log x)^{7/8}\right) \right )  \]
for some suitable $c>0$.

The following localised version of the above theorem is immediate from its proof.
\begin{coro}\label{coro:balu_ramachandra1}
For any $c>0$ and for all sufficiently large $T$ depending on $c$,
there exists an 
\[ X \in \left[ T, \frac{T^{\mathfrak b}}{2}\log^2 T\right]\] 
for which we have
\[ \int_X^{2X} \frac{|\Delta(x)|^2}{x^{2\alpha+1}} dx 
\ge \exp\left( (c-\epsilon)(\log X)^{7/8}\right),\]
with $\alpha$ as in Theorem \ref{omega_integral} and for any $\epsilon>0$.
\end{coro}
Optimality of the above bound is justified in Proposition~\ref{prop:upper_bound_second_moment_twisted_divisor}. 
We also prove a \lq measure version \rq \ of this result:
\begin{thm}\label{Balu-Ramachandra-measure}
For any $c>0$, let
\[\alpha(x)=\frac{3}{8}- \frac{c}{(\log x)^{1/8} }\] and  
$\A=\{x: |\Delta(x)|\gg x^{\alpha(x)} \}$. 
Then 
\[ \mu(\A\cap [X,2X])=\Omega(X^{2\alpha(X)}), \ \text{ as } X\rightarrow \infty.\]
\end{thm}

\section{Prerequisites}

In order to prove the theorem, we need several lemmas, which form the content of this section. 
We begin with a fixed $\delta_0 \in (0,1/16]$ for which we would choose a numerical value at the end of this section.

\begin{defi}
 For $T>1$, let $Z(T)$ be the set of all $\gamma$ such that 
\begin{enumerate}
\item $T\le \gamma \le 2T$,
\item either $\zeta(\beta_1+i\gamma)=0$ for some $\beta_1\ge \half+\frac{\delta_0}{2}$ \\
or $\zeta(\beta_2+i 2\gamma)=0$ for some $\beta_2\ge \half +\frac{\delta_0}{2}$.
\end{enumerate}
Let
\[ I_{\gamma,k} = \{ T\le t \le 2T: |t-\gamma| \le k\log^2 T \}  \text{ for }  k=1, 2.\]
We finally define
\[J_k(T)=[T,2T] \setminus \cup_{\gamma\in Z(T)} I_{\gamma,k}. \]
\end{defi}

\begin{lem}\label{size-J(T)}
With the above definition, we have for $k=1,2$
\[ \mu(J_k(T)) = T +O\left( T^{1-\delta_0/8} \log^3 T \right). \]
\end{lem}

\begin{proof}
We shall use an estimate on the function $N(\sigma, T)$, which is defined as 
\[N(\sigma, T):=\left|\{\sigma'+it:\sigma'\ge\sigma,\ 0<t\leq T,\ \zeta(\sigma'+it)=0\}\right|.\]
Selberg \cite[Page~237]{Titchmarsh} proved that
$$N(\sigma, T) \ll T^{1-\frac{1}{4}(\sigma -\half)} \log T, \ \text{ for } \ \sigma>1/2.$$
%for $\sigma>1/2$. 
Now the lemma follows from the above upper bound on $N(\sigma, t)$, and the observation that
$$\mu\left(\cup_{\gamma\in Z(T)} I_{\gamma,k}\right) \ll N\left(\half+ \frac{\delta_0}{2}, T\right)\log^2 T.$$
\end{proof}

The next lemma closely follows Theorem 14.2 of  \cite{Titchmarsh}, but we are including a proof 
as we could not find a clearly written proof of this version which unlike the original one,
does not use Riemann Hypothesis.
\begin{lem}\label{estimate-on-J(T)}
For $t\in J_1(T)$ and $\sigma= 1/2+\delta$ with $\delta_0 < \delta < 1/4-{\delta_0}/2$,  we have
$$|\zeta(\sigma+it)|^{\pm 1} \ll 
\exp\left(\log\log t \left(\frac{\log t}{\delta_0}\right)^{\frac{1-2\delta}{1-2\delta_0}}\right)$$

and
\[ |\zeta(\sigma+2it)|^{\pm 1} \ll
\exp\left(\log\log t \left(\frac{\log t}{\delta_0}\right)^{\frac{1-2\delta}{1-2\delta_0}}\right).\]
\end{lem}

\begin{proof}
We provide a proof of the first statement, and the second statement can be similarly proved.

Let $1 <\sigma' \le \log t$. We consider two concentric circles centered at $\sigma'+it$,
with radius $\sigma'-1/2-\delta_0/2$ and $\sigma' -1/2- \delta_0$. 
Since $t\in J_1(T)$ and the radius of the circle is $\ll \log t$, we conclude that
\[ \zeta(z)\neq 0 \ \text{ for } \  |z-\sigma'-it | \le \sigma' - \half -  \frac{\delta_0}{2} \]
and also $\zeta(z)$ has polynomial growth in this region. 
Thus on the larger circle, $\log |\zeta(z)| \le c_5\log t$, for some constant $c_5>0$.
By Borel-Caratheodory theorem, 
\[\  |z-\sigma'-it | \le \sigma' - \half - \delta_0 \  \text{ implies } \
|\log \zeta(z)| \le \frac{c_6\sigma'}{\delta_0} \log t, \]
for some $c_6>0$.
Let $1/2+\delta_0< \sigma < 1$, and $\xi>0$ be such that $1+\xi< \sigma'$. 
We consider three concentric circles centered at 
$\sigma'+it$ with radius $r_1=\sigma'-1-\xi$, $r_2=\sigma'-\sigma$ and 
$r_3=\sigma'-1/2-\delta_0$,
and call them $\mathcal  C_1, \mathcal C_2$ and $\mathcal C_3$ respectively. 
Let
$$M_i = \sup_{z\in \mathcal C_i} |\log \zeta(z)|.$$
From the above bound on $|\log\zeta(z)|$, we get 
$$M_3 \le  \frac{c_6\sigma'}{\delta_0} \log t.$$
Suitably enlarging $c_6$, we see that 
\[ M_1 \le \frac{c_6}{\xi}.\]
Hence we can apply the Hadamard's three circle theorem to conclude that
\[ M_2 \le M_1^{1-\nu} M_3^\nu,  \ \text{ for } \ \nu=\frac{\log(r_2/r_1)}{\log(r_3/r_1)}. \]
Thus
\[ M_2 \le \left( \frac{c_6}{\xi} \right)^{1-\nu}\left(\frac{c_6\sigma' \log t}{\delta_0}\right)^\nu. \]
It is easy to see that
\[ \nu=2-2\sigma + \frac{4\delta_0(1-\sigma)}{1+2\xi-2\delta_0} 
+O(\xi) + O\left( \frac{1}{\sigma'}\right). \]
Now we put 
\[ \xi=\frac{1}{\sigma'}=\frac{1}{\log\log t} .\]
Hence
\[ M_2 \le \frac{c_6  \log^\nu t \log\log t}{\delta_0^\nu}
=\frac{c_7 \log\log t}{\delta_0^\nu}(\log t)^{2-2\sigma+\frac{4\delta_0(1-\sigma)}{1+2\xi-2\delta_0} }, \]
for some $c_7>0$.
We observe that
\[  2-2\sigma+\frac{4\delta_0(1-\sigma)}{1+2\xi-2\delta_0} < 
2-2\sigma+\frac{4\delta_0(1-\sigma)}{1-2\delta_0} =\frac{1-2\delta}{1-2\delta_0}. \]
So we get 
\[ |\log \zeta(\sigma +it) | 
\le c_7 \log\log t \left(\frac{\log t}{\delta_0}\right)^{\frac{1-2\delta}{1-2\delta_0}},\]
and hence the lemma.
\end{proof}

We put $y=T^{\mathfrak b}$, for a constant $\mathfrak b \ge 80$. Now suppose that
$$\int_T^{\infty} \frac{|\Delta(u)|^2}{u^{2\alpha+1}}e^{-u/y}\d u \ge \log^2 T,$$
for sufficiently large $T$. Then clearly 
$$\Delta(u) =\Omega( u^{\alpha}) .$$ 
Our next result explores the situation when such an inequality does not hold.

\begin{prop}\label{main-prop}
Let $\delta_0<\delta<\frac{1}{4}-\frac{\delta_0}{2}$.
For $1/4+\delta/2 < \alpha <1/2$, suppose that
 \begin{equation}\label{assumption}
\int_T^{\infty} \frac{|\Delta(u)|^2}{u^{2\alpha+1}}e^{-u/y}\d u \le \log^2 T,
\end{equation}
for a sufficiently large $T$.
Then we have
$$\int_{J_2(T)} \frac{|D(\alpha+it)|^2}{|\alpha+it|^2}\d t 
\ll 
1 + \int_T^{\infty} \frac{|\Delta(u)|^2}{u^{2\alpha+1}}e^{-2u/y} \d u.$$
\end{prop}

Before embarking on a proof, we need the following lemma which is easy to prove using 
Stirling's formula for $\Gamma$-function.

\begin{lem}\label{gamma}
Let $z$ be a complex number with $0\le \Re(z) \le 1$ and $|\Im(z)|\ge \log^2T$. For $y$ as above, we have
\begin{equation}\label{gamma1}
\int_T^{\infty} e^{-u/y}u^{-z} \d u =\frac{T^{1-z}}{1-z} + O(T^{-\mathfrak b'})
\end{equation}
and
\begin{equation}\label{gamma2}
\int_T^{\infty} e^{-u/y} u^{-z}\log u\ \d u =\frac{T^{1-z}}{1-z}\log T + O(T^{-\mathfrak b'}),
\end{equation}
where $\mathfrak b'>0$ depends only on $\mathfrak b$.
\end{lem}

\begin{lem}\label{initial-estimates}
Under the assumption (\ref{assumption}), there exists $T_0$ with $T\le T_0 \le 2T$ such that
\begin{equation*}
\frac{\Delta(T_0)e^{-T_0/y}}{T_0^{\alpha}} \ll \log^2 T,
\end{equation*}
\begin{equation*}
\text{and}\quad\frac{1}{y}\int_{T_0}^{\infty}\frac{\Delta(u)e^{-u/y}}{u^{\alpha}} \d u \ll \log T.
\end{equation*}

\end{lem}

\begin{proof}

The assumption (\ref{assumption}) implies that
\begin{eqnarray*}
\log^2T &\ge & \int_T^{2T} \frac{|\Delta(u)|^2}{u^{2\alpha+1}}e^{-u/y}\d u 
= \int_T^{2T} \frac{|\Delta(u)|^2}{u^{2\alpha}}e^{-2u/y}\frac{e^{u/y}}{u} \d u \\
&\ge & \min_{T\le u\le 2T}\left(\frac{|\Delta(u)|}{u^{\alpha}}e^{-u/y}\right)^2, 
\end{eqnarray*}
which proves the first assertion.
To prove the second assertion, we use the previous assertion and Cauchy- Schwartz inequality along with assumption (\ref{assumption}) to get
\begin{eqnarray*}
\left( \int_{T_0}^{\infty}\frac{|\Delta(u)|}{u^{\alpha}}e^{-u/y}\d u \right)^2
&\le & \left( \int_{T_0}^{\infty}\frac{|\Delta(u)|^2}{u^{2\alpha+1}}e^{-u/y}\d u \right)
\left(  \int_{T_0}^{\infty} u e^{-u/y}\d u \right) \\
&\ll & y^2 \log^2 T.
\end{eqnarray*}
This completes the proof of this lemma.
\end{proof}

We now recall a mean value theorem due to Montgomery and Vaughan \cite{MontgomeryVaughan}.
\begin{notation}
 For a real number $\theta$, let $\|\theta\|:=\min_{n\in \mathbb Z}|\theta -n|.$
\end{notation}

\begin{thm}[Montgomery and Vaughan \cite{MontgomeryVaughan}]\label{mean-value}
Let $a_1,\cdots, a_N$ be arbitrary complex numbers, and let $\lambda_1,\cdots,\lambda_N$ be distinct real numbers such that 
\[\delta = \min_{\substack{m,n\\ m\neq n}}\| \lambda_m-\lambda_n\|>0.\]
Then
\[ \int_0^T \left| \sum_{n\le N} a_n \exp(i\lambda_n t) \right|^2 \d t 
=\left(T +O\left(\frac{1}{\delta}\right)\right)\sum_{n\le N} |a_n|^2.\]
\end{thm}

\begin{lem}\label{mean-value-estimate}
For $T\le T_0\le 2T$ and $\Re(s)=\alpha$, we have
$$\int_T^{2T} \left| \sum_{n\le T_0}\frac{|\tau(n,\theta)|^2}{n^s}e^{-n/y}\right|^2 t^{-2} dt \ll 1.$$
\end{lem}

\begin{proof}
Using  theorem \ref{mean-value}, we get
\begin{eqnarray*}
&&\int_T^{2T} \left| \sum_{n\le T_0}\frac{|\tau(n,\theta)|^2}{n^s}e^{-n/y}\right|^2 t^{-2} dt \\
&\le & \frac{1}{T^2} \left( T \sum_{n\le T_0}  |b(n)|^2 
+ O\left( \sum_{n\le T_0} n|b(n)|^2\right)\right),
\end{eqnarray*}
where 
\[ b(n)=\frac{|\tau(n,\theta)|^2}{n^{\alpha}}e^{-n/y}.\]
Thus
\[ \sum_{n\le T_0}  |b(n)|^2 \le \sum_{n\le T_0}\frac{d(n)^4}{n^{2\alpha}}
\ll T_0^{1-2\alpha+\epsilon}\]
and 
\[ \sum_{n\le T_0}  n|b(n)|^2 \le \sum_{n\le T_0}\frac{d(n)^4}{n^{2\alpha-1}}
\ll T_0^{2-2\alpha + \epsilon}\]
for any $\epsilon>0$, since the divisor function $d(n)\ll n^\epsilon$. As we have $\alpha>0$, this completes the proof.
\end{proof}

\begin{lem}\label{mean-value-error}
For $\Re(s)=\alpha$ and $T\le T_0 \le 2T$, we have
$$\int_T^{2T} 
\left| \sum_{n\ge 0}\int_0^1 \frac{\Delta(n+x+T_0) 
e^{-(n+x+T_0)/y}}{(n+x+T_0)^{s+1}} \d x \right|^2 \d t
\ll \int_T^{\infty} \frac{|\Delta(x)|^2}{x^{2\alpha+1}}e^{-2x/y} \d x.$$
\end{lem}

\begin{proof}  
Using Cauchy- Schwarz inequality, we get
\begin{align*}
 & \left| \sum_{n\ge 0}\int_0^1  \frac{\Delta(n+x+T_0)}{(n+x+T_0)^{s+1}}
e^{-(n+x+T_0)/y} \d x \right|^2 \\
\le& \int_0^1 \left| \sum_{n\ge 0} \frac{\Delta(n+x+T_0)}{(n+x+T_0)^{s+1}}e^{-(n+x+T_0)/y} \right|^2 \d x.
\end{align*}
Hence 
\begin{align*}
&\int_T^{2T} 
\left| \int_0^1 \sum_{n\ge 0}\frac{\Delta(n+x+T_0) e^{-(n+x+T_0)/y}}{(n+x+T_0)^{s+1}} \d x \right|^2 \d t \\
\le&  \int_T^{2T}\int_0^1 \left| \sum_{n\ge 0} \frac{\Delta(n+x+T_0)}{(n+x+T_0)^{s+1}}
e^{-(n+x+T_0)/y} \right|^2 \d x \d t \\
=& \int_0^1 \int_T^{2T}\left| \sum_{n\ge 0} \frac{\Delta(n+x+T_0)}{(n+x+T_0)^{s+1}}
e^{-(n+x+T_0)/y} \right|^2 \d t \d x.
\end{align*}
From Theorem \ref{mean-value}, we can get
\begin{eqnarray*}
&&\int_T^{2T}\left| \sum_{n\ge 0} \frac{\Delta(n+x+T_0)}{(n+x+T_0)^{s+1}}e^{-(n+x+T_0)/y} \right|^2 \d t\\
&=& T\sum_{n\ge 0}\frac{ |\Delta(n+x+T_0)|^2}{(n+x+T_0)^{2\alpha+2}}
e^{-2(n+x+T_0)/y} 
+ O\left( \sum_{n\ge 0} \frac{ |\Delta(n+x+T_0)|^2}{(n+x+T_0)^{2\alpha+1}}e^{-2(n+x+T_0)/y}\right)\\
&\ll & \sum_{n\ge 0} \frac{ |\Delta(n+x+T_0)|^2}{(n+x+T_0)^{2\alpha+1}}e^{-2(n+x+T_0)/y}.
\end{eqnarray*}
Hence
\begin{eqnarray*}
&&\int_T^{2T} 
\left| \sum_{n\ge 0}\int_0^1 \frac{\Delta(n+x+T_0) e^{-(n+x+T_0)/T}}{(n+x+T_0)^{s+1}} \d x \right|^2 \d t \\
&\ll & \int_0^1 \sum_{n\ge 0} \frac{ |\Delta(n+x+T_0)|^2}{(n+x+T_0)^{2\alpha+1}}e^{-2(n+x+T_0)/y}\d x
\ll  \int_T^{\infty} \frac{|\Delta(x)|^2}{x^{2\alpha+1}}e^{-2x/y} \d x,
\end{eqnarray*}
completing the proof.
\end{proof}

\begin{proof}[\textbf{Proof of Proposition \ref{main-prop}.} ]
For $s=\alpha+it$ with $1/4 +\delta  < \alpha < 1/2$ and $t\in J_2(T)$, we have
\begin{eqnarray*}
\sum_{n=1}^\infty \frac{|\tau(n,\theta)|^2}{n^s} e^{-n/y}
&=& \frac{1}{2\pi i} \int_{2-i\infty}^{2+i\infty}
D(s+w) \Gamma(w) y^w \d w \\
&=& \frac{1}{2\pi i} \int_{2-i\log^2T}^{2+i\log^2T}
+O\left( y^2\int_{\log^2T}^{\infty} |D(s+2+iv)||\Gamma(2+iv)|\d v \right).
\end{eqnarray*}
The above error term is estimated to be $o(1)$. We move the integral to 
$$\left[\frac{1}{4}+\frac{\delta}{2} -\alpha-i\log^2T, 
\ \frac{1}{4}+\frac{\delta}{2} -\alpha+i\log^2 T\right].$$
Let $\delta'=1/4+\delta/2 -\alpha$.
In the region to the right side of this line, $\Re(2s+2w)\ge 1/2+\delta.$  Writing $w=u+iv$ we observe that $t+v \in J_1(T)$ since 
$t\in J_2(T)$. So we can apply 
Lemma~\ref{estimate-on-J(T)} to 
conclude that
$$\zeta(2s+2w)\gg T^{-1}.$$ 
On the above line, we have $\Re(s+w)=1/4+\delta/2$, Thus
$$\zeta^2(s+w)\zeta(s+w+i\theta)\zeta(s+w-i\theta) \ll T^{3/2-\delta}\log^4 T,$$
where we use the fact that $\zeta(z)\ll \Im(z)^{{(1-\Re(z))}/2}\log(\Im(z))$ if $0\le \Re(z) \le 1$.
Hence by convexity, we see that $ \zeta^2(s+w)\zeta(s+w+i\theta)\zeta(s+w-i\theta)$ has polynomial growth 
on the horizontal lines of integration.  Therefore the horizontal integrals are $o(1)$ by exponential decay of $\Gamma$-function.
Since the only pole inside this contour is at $w=0$, we get
\begin{eqnarray*}
\sum_{n=1}^{\infty} \frac{|\tau(n,\theta)|^2}{n^s} e^{-n/y} 
= D(s) + \frac{1}{2\pi i}
\int_{\delta'-i\log^2 T}^{\delta'+i\log^2 T}
D(s+w)\Gamma(w)y^w \d w +o(1).
\end{eqnarray*}
For the integral on the right hand side, we have
\[ D(s+w) y^w \ll T^{5/2 -\delta(\mathfrak b/2 +1)}\]
where the exponent of $T$ is negative by our choice of $\mathfrak b$ and $\delta$.
Therefore this integral is also $o(1)$. 

Using $T_0$ as in Lemma~\ref{initial-estimates}, we now divide the sum into two parts:
$$D(s)= \sum_{n\le T_0} \frac{|\tau(n, \theta)|^2}{n^s} e^{-n/y} 
+ \sum_{n > T_0} \frac{|\tau(n, \theta)|^2 }{n^s} e^{-n/y}+o(1).$$
To estimate the second sum, we write
\begin{eqnarray*}
\sum_{n > T_0} \frac{|\tau(n, \theta)|^2}{n^s} e^{-n/y} 
&=& \int_{T_0}^{\infty} \frac{ e^{-x/y}}{x^s} \d \left(\sum_{n\le x}|\tau(n, \theta)|^2 \right)\\
 &=& \int_{T_0}^{\infty} \frac{ e^{-x/y}}{x^s} \d (\M(x)+\Delta(x))\\
&=&  \int_{T_0}^{\infty} \frac{ e^{-x/y}}{x^s} \M'(x)\d x 
   + \int_{T_0}^{\infty} \frac{ e^{-x/y}}{x^s} \d (\Delta(x)).
\end{eqnarray*}
Recall that
\[ \M(x)=\omega_1(\theta)x\log x + \omega_2(\theta)x\cos(\theta\log x)
+\omega_3(\theta)x,\]
thus 
\[ \M'(x)=\omega_1(\theta)\log x + \omega_2(\theta)\cos(\theta\log x)
-\theta\omega_2(\theta)\sin(\theta\log x)+\omega_1(\theta)+\omega_3(\theta).\]
 Observe that
\[ \int_{T_0}^{\infty}\frac{ e^{-x/y}}{x^s}\cos(\theta\log x) \d x
=\half \int_{T_0}^{\infty} \frac{ e^{-x/y}}{x^{s+i\theta}} \d x
+\half \int_{T_0}^{\infty} \frac{ e^{-x/y}}{x^{s-i\theta}} \d x.\]
Applying Lemma~\ref{gamma}, we conclude that 
\[ \int_{T_0}^{\infty} \frac{ e^{-x/y}}{x^s} \M'(x)\d x = o(1).\]
Integrating the second integral by parts:
\begin{eqnarray*}
\int_{T_0}^{\infty} \frac{ e^{-x/y}}{x^s} \d (\Delta(x)) 
&=& \frac{e^{-T_0/y} \Delta(T_0)}{T_0^s} \\
&+& \frac{1}{y}\int_{T_0}^\infty \frac{ e^{-x/y}}{x^s}\Delta(x)  \d x
-s\int_{T_0}^{\infty}\frac{ e^{-x/y}}{x^{s+1}}\Delta(x) \d x.
\end{eqnarray*}
Applying Lemma~\ref{initial-estimates}, we get
\begin{eqnarray*}
\sum_{n > T_0} \frac{|\tau(n, \theta)|^2}{n^s} e^{-n/y} &=&
s\int_{T_0}^{\infty}\frac{\Delta(x) e^{-x/y}}{x^{s+1}} \d x  + O(\log T) \\
&=& s\sum_{n\ge 0} \int_{0}^{1}\frac{\Delta(n+x+T_0) e^{-(n+x+T_0)/y}}{(n+x+T_0)^{s+1}} \d x +O(\log T).
\end{eqnarray*}
Hence we have
$$D(s)= \sum_{n\le T_0} \frac{|\tau(n, \theta)|^2}{n^s} e^{-n/y} 
+s\sum_{n\ge 0} \int_{0}^{1}\frac{\Delta(n+x+T_0) e^{-(n+x+T_0)/y}}{(n+x+T_0)^{s+1}} \d x +O(\log T) .$$
Squaring both sides, and then integrating on $J_2(T)$, we get
\begin{align*}
\int_{J_2(T)} \frac{|D(\alpha+it)|^2}{|\alpha+it|^2} \d t 
& \ll \int_T^{2T} \left|  \sum_{n\le T_0} \frac{|\tau(n, \theta)|^2}{n^s} 
  e^{-n/y} \right|^2 \frac{ \d t}{t^2} \\
& + \int_T^{2T} 
\left| \sum_{n\ge 0} \int_{0}^{1}\frac{\Delta(n+x+T_0) e^{-(n+x+T_0)/y}}{(n+x+T_0)^{s+1}} \d x \right|^2 \d t.
\end{align*}
The proposition now follows using Lemma~\ref{mean-value-estimate} and Lemma~\ref{mean-value-error}.
\end{proof}
%We are now ready to prove our main theorem of this section.
\section{Proofs of The Main Theorems}
\subsection{Proof of Theorem \ref{omega_integral}}
%\begin{proof}[\textbf{Proof of Theorem \ref{omega_integral}}]
We prove by contradiction. Suppose that (\ref{lb-increasing}) does not hold.
Then there exists a constant $c>0$ such that given any $N_0>1$, there exists $T>N_0$ for which
\[\int_T^{\infty} \frac{|\Delta(x)|^2}{x^{2\alpha+1}}e^{-2x/y} \d x 
\leq \exp\left( c(\log T)^{7/8} \right).\]
Note that the above statement is weaker than the contrapositive of the statement of theorem. 
This gives
\[ \int_T^{\infty} \frac{|\Delta(x)|^2}{x^{2\beta+1}}e^{-2x/y} \d x \ll 1, \]
where
\[
\beta=\frac{3}{8}-\frac{c}{2(\log T)^{1/8}} .
\]
We apply Proposition~\ref{main-prop} to get 
\begin{equation}\label{contra}
\int_{J_2(T)} \frac{|D(\beta+it)|^2}{|\beta+it|^2} \d t 
\ll 1.
\end{equation}
Now we compute a lower bound for the last integral over $J_2(T)$.
Write the functional equation for $\zeta(s)$ as
$$\zeta(s) =\pi^{1/2-s}\frac{\Gamma((1-s)/2)}{\Gamma(s/2)}\zeta(1-s).$$
Using the Stirling's formula for $\Gamma$ function, we get
$$|\zeta(s)|=\pi^{1/2-\sigma}t^{1/2-\sigma}|\zeta(1-s)|\left(1+O\left(\frac{1}{T}\right)\right)$$
for $s=\sigma+it$. This implies
$$|D(\beta+it)|=t^{2-4\beta}\frac{|\zeta(1-\beta+it)^2\zeta(1-\beta-it-i\theta)
\zeta(1-\beta-it+i\theta)|}{|\zeta(2\beta+i2t)|}.$$
Let $\delta_0=1/16$, and 
\[\beta=\frac{3}{8} -\frac{c}{2(\log T)^{1/8} }=\half-\delta \]
with 
\[\delta=\frac{1}{8}+\frac{c}{2(\log T)^{1/8} }.\]
Then using Lemma~\ref{estimate-on-J(T)}, we get
\[ |\zeta(1-\beta+it)| = \left|\zeta\left(\half+\delta+it\right)\right|
\gg \exp\left(\log\log t \left(\frac{\log t}{\delta_0}\right)^{\frac{1-2\delta}{1-2\delta_0}}\right).\]
For $t\in J_2(T)$ we observe that $t\pm\theta \in J_1(T)$, and so the same
bounds hold for $\zeta(1-\beta+it+i\theta)$ and $\zeta(1-\beta+it -i\theta)$. 
Further
\[ |\zeta(2\beta+i2t)| = \left|\zeta\left(\half+\left(\half-2\delta\right)+i2t\right)\right|
\ll \exp\left(\log\log t \left(\frac{\log t}{\delta_0}\right)^{\frac{4\delta}{1-2\delta_0}}\right).\]
Combining these bounds, we get
\[ |D(\beta+it)| \gg t^{2-4\beta}
\exp\left(-5\log\log t \left(\frac{\log t}{\delta_0}\right)^{\frac{1-2\delta}{1-2\delta_0}}\right).\]
Therefore
\begin{eqnarray*}
\int_{J_2(T)} |D(\beta+it)|^2 \d t 
&\gg & T^{4-8\beta}
 \exp\left(-10\log\log T\left(\frac{\log T}{\delta_0}\right)^{\frac{1-2\delta}{1-2\delta_0}}\right)
\mu(J_2(T)) \\
&\gg & T^{5-8\beta}\exp\left(-10\log\log T \left(\frac{\log T}{\delta_0}\right)^{\frac{1-2\delta}{1-2\delta_0}}\right),
\end{eqnarray*}
where we use Lemma  \ref{size-J(T)} to show that $\mu(J_2(T))\gg T$.
Now putting the values of  $\delta$ and $\delta_0$ as chosen above, we get
$$\int_{J_2(T)} \frac{|D(\beta+it)|^2}{|\beta+it|^2} dt 
\gg \exp\left(3c(\log T)^{7/8}\right),$$
since $\frac{1-2\delta}{1-2\delta_0}< 7/8$. This contradicts (\ref{contra}), and hence the theorem follows.
%\end{proof}

\subsection{Proof of Theorem \ref{Balu-Ramachandra-measure}}
%\begin{proof}[\textbf{Proof of Theorem \ref{Balu-Ramachandra-measure}}] 
Suppose that the conclusion does not hold, hence
\[  \mu(\A\cap [X,2X]) \ll X^{2\alpha(X)}.\] 
Thus for every sufficiently large $X$, we get
\[ \int_{\A\cap [X,2X] }\frac{|\Delta(x)|^2}{x^{2\alpha+1}}dx 
\ll X^{2\alpha}\frac{ M(X)}{X^{2\alpha+1}}=\frac{ M(X)}{X},\]
where $\alpha=\alpha(X)$ and $M(X)=\sup_{X\le x \le 2X} |\Delta(x)|^2$.
Using dyadic partition, we can prove
\[ \int_{\A\cap [T,y] }\frac{|\Delta(x)|^2}{x^{2\alpha+1}}dx \ll \frac{M_0(T)}{T}\log T, \
\text{ where }
\ M_0(T) =\sup_{T\le x\le y} |\Delta(x)|^2 \]
and $y=T^{\mathfrak b}$ for some $\mathfrak b>0$ and $T$ sufficiently large. 
This gives
\[ \int_T^{\infty} \frac{|\Delta(x)|^2}{x^{2\alpha+1}}e^{-2x/y} dx
\ll \frac{M_0(T)}{T}\log T. \]
Along with (\ref{lb-increasing}), this implies
\[ M_0(T)\gg T\exp\left( \frac{c}{2}(\log T)^{7/8} \right).\]
Thus
\[|\Delta(x)| \gg x^{\half} \exp\left(\frac{c}{4}(\log x)^{7/8}\right),\]
for some $x\in [T,y].$
This contradicts the fact that $|\Delta(x)| \ll x^{\half} (\log x)^6.$
%\end{proof}
%%%begin{defi}
%%% An infinite unbounded subset $\set$ of non-negative real numbers is called an $\xset$ . 
%%%\end{defi}
\section{Optimality of the Omega Bound for the Second Moment }
The following proposition shows the optimality of the omega bound in Corollary~\ref{coro:balu_ramachandra1}.
\begin{prop}\label{prop:upper_bound_second_moment_twisted_divisor}
 Under Riemann Hypothesis (RH), we have
 \[\int_{X}^{2X}\Delta^2(x)\d x\ll X^{7/4+\epsilon}\]
for any $\epsilon>0$. 
\end{prop}
\begin{proof}
Perron's formula gives
\begin{equation*}
\Delta(x)=\frac{1}{2\pi i}\int_{-T}^{T}\frac{D(3/8+it)x^{3/8+it}}{3/8+it}\d t + O(x^\epsilon),
\end{equation*}
for any $\epsilon>0$ and for $T=X^2$ with $x\in[X, 2X]$. Using this expression for $\Delta(x)$, we write its second moment as
\begin{align*}
&\int_{X}^{2X}\Delta^2(x)\d x 
= \frac{1}{(2\pi)^2}\int_{X}^{2X}\int_{-T}^{T}\int_{-T}^{T}\frac{D(3/8+ it_1)D(3/8- it_2)}{(3/8+it_1)(3/8- it_2)}x^{3/4+ i(t_1-t_2)}\d x \ \d t_1 \d t_2 \\
&\hspace{2.5 cm}  
+ O\left(X^{1+\epsilon}(1+|\Delta(x)|)\right)\\
&\ll X^{7/4}\int_{-T}^{T}\int_{-T}^{T}\left|\frac{D(3/8+ it_1)D(3/8- it_2)}{(3/8+it_1)(3/8-it_2)(7/4+ i(t_1-t_2))}\right|\d t_1 \d t_2 
+ O(X^{3/2+\epsilon}).\\
\end{align*}
In the above calculation, we have used the fact that $\Delta(x)\ll x^{\half+\epsilon}$ as in (\ref{eq:upper_bound_delta}). Also note that for 
complex numbers $a, b$, we have $|ab|\leq \half(|a|^2 + |b|^2)$. We use this inequality with
\[a=\frac{|D(3/8+it_1)|}{|3/8+it_1|\sqrt{|7/4+i(t_1-t_2)|}} \ \text{ and } 
\ b=\frac{|D(3/8-it_2)|}{|3/8-it_2|\sqrt{|7/4+i(t_1-t_2)|}},\]
to get
\begin{align*}
\int_{X}^{2X}\Delta^2(x)\d x
&\ll X^{7/4}\int_{-T}^{T}\int_{-T}^{T}\left|\frac{D(3/8- it_2)}{(3/8-it_2)}\right|^2\frac{\d t_1}{|7/4+ i(t_1-t_2)|} \d t_2 + O(X^{3/2+\epsilon})\\
&\ll X^{7/4}\log X\int_{-T}^{T}\left|\frac{D(3/8- it_2)}{(3/8-it_2)}\right|^2 \d t_2 + O(X^{3/2+\epsilon}).
\end{align*}
Under RH, convexity bound gives $\zeta(\sigma+it) \ll t^{1/2-\sigma}$ for $0 \le \sigma \le 1/2$,
hence $|D(3/8-it_2)|\ll |t_2|^{\half+\epsilon}$. So we have 
\begin{align*}
 \int_{X}^{2X}\Delta^2(x)\d x 
 \ll X^{7/4+\epsilon} \ \text{ for any } \ \epsilon>0.
\end{align*}
\end{proof}
\section*{Acknowledgement}
We thank R. Balasubramanian and K. Srinivas for many pertinent comments.
K. Mahatab is supported by Grant 227768 of the Research Council of Norway, and this work was carried out when he was a research fellow at the Institute of Mathematical Sciences, Chennai.

\bibliographystyle{abbrv}
\bibliography{refs_omega}

\end{document}